\documentclass{article}
\usepackage[%
]{ems-journal}

\numberwithin{equation}{section}
\usepackage{graphicx} 
\graphicspath{{./Images/}}
\usepackage{sectsty}
\usepackage{amsmath}
\usepackage[most]{tcolorbox}
\usepackage{bbm}
\usepackage{bigints}
\usepackage{tikz}
\usetikzlibrary{calc}
\usepackage{lipsum}
\usepackage{subcaption}
\usepackage{pythonhighlight}
\usepackage{setspace,mathrsfs,amsthm,mathtools,cite,tabularx,fontenc}
\usepackage{appendix}
\usepackage{enumitem}
\newtcolorbox{mybox}[1]{%
    tikznode boxed title,
    enhanced,
    arc=0mm,
    interior style={white},
    attach boxed title to top left= {yshift=-\tcboxedtitleheight/2-0.05cm, xshift=0.7cm},
    fonttitle=\small\bfseries,
    colbacktitle=white,coltitle=black,
    boxed title style={size=small,colframe=white,boxrule=0pt},
    title={#1}}

\makeatletter
\newtheorem*{rep@theorem}{\rep@title}
\newcommand{\newreptheorem}[2]{%
	\newenvironment{rep#1}[1]{%
		\def\rep@title{#2 \ref{##1}}%
		\begin{rep@theorem}}%
		{\end{rep@theorem}}}
\makeatother

\newtheorem{theorem}{Theorem}[subsection]
\newreptheorem{theorem}{Theorem}
\newtheorem{lemma}{Lemma}[subsection]

\newtheorem{proposition}[lemma]{Proposition}
\theoremstyle{definition}
\newtheorem{remark}{Remark}[subsection]
\newtheorem{definition}{Definition}[subsection]

\usepackage[nointegrals]{wasysym}
\usepackage[misc]{ifsym}
\usepackage[dvipsnames]{xcolor}
\definecolor{ao}{rgb}{0.0, 0.5, 0.0}
\definecolor{lasallegreen}{rgb}{0.03, 0.47, 0.19}
\usepackage{hyperref}

\newcommand{\AM}[1]{\textcolor{darkgreen}{#1}} 
\definecolor{darkgreen}{rgb}{0,0.7,0}

\let\oldnorm\norm
\def\norm{\@ifstar{\oldnorm}{\oldnorm*}}

\newcommand{\pa} {\partial}

\newcommand{\De} {\Delta}

\newcommand{\Om} {\Omega}

\newcommand{\si} {\sigma}

\newcommand{\Gr} {\nabla}

\newcommand{\sub}{\subset}

\newcommand{\dx}{{\,\textrm d}x}
\newcommand{\dy}{{\,\textrm d}y}

\newcommand{\dsi}{{\,\textrm d}\si}

\newcommand{\F}{{\mathcal F}}

\newcommand{\R}{{\mathbbm R}}
\newcommand{\N}{{\mathbb N}}

\def\dx{{\,\textrm d}x}

\def\F{{\mathcal F}}

\begin{document}

\title{A new convergence proof for an iterative scheme for elliptic Cauchy problems based on alternating orthogonal projections}
\titlemark{Iterative scheme for elliptic Cauchy problems}


%

\emsauthor{1}{
	\givenname{Abhijit}
	\surname{Guchhait}
	\mrid{}
	\zblid{}
	\orcid{}}{A.~Guchhait}
\emsauthor{2}{
	\givenname{Adrian}
	\surname{Muntean}
	\mrid{}
	\zblid{}
	\orcid{}}{A.~Muntean}
\emsauthor{3}{
	\givenname{Martin}
	\surname{Lind}
	\mrid{}
	\zblid{}
	\orcid{}}{M.~Lind}

\Emsaffil{1}{
	\department{Department of Mathematics and Computer Science}
	\organisation{Karlstad University}
	\rorid{}
	\address{}
	\zip{}
	\city{}
	\country{Sweden}
	\affemail{abhijit.guchhait@kau.se}}
\Emsaffil{2}{
	\department{1}{Department of Mathematics and Computer Science}
	\organisation{1}{Karlstad University}
	\rorid{1}{}
	\address{1}{}
	\zip{1}{}
	\city{1}{}
	\country{1}{Sweden}
	\affemail{adrian.muntean@kau.se}}
\Emsaffil{3}{
	\department{Department of Mathematics and Computer Science}
	\organisation{Karlstad University}
	\rorid{}
	\address{}
	\zip{}
	\city{}
	\country{Sweden}
	\affemail{martin.lind@kau.se}}

\classification[35J25, 47A50]{35R25, 73C35}


\begin{abstract} 
We consider an elliptic Cauchy problem for the Poisson equation in which both Dirichlet and Neumann boundary data are prescribed only on a portion of the boundary, while no boundary conditions are available on the remaining part of the boundary. Such problems are severely ill-posed. 
In this work, we study the convergence of an iterative method based on the alternating imposition of Dirichlet and Neumann boundary conditions.
We formulate the scheme within a Hilbert space framework, where a yet another alternation becomes apparent: each iteration can now be interpreted as an orthogonal projection onto an appropriate closed subspace. This perspective provides a transparent geometric interpretation of the algorithm and yields a concise convergence proof based on von Neumann's classical result on alternating (orthogonal) projections. In addition to elucidating the convergence of the method, this alternating projections perspective also offers a general rate of convergence in terms of the Friedrichs angle (between two closed subspaces). We point out by means of an example that such convergence rate estimate degenerates for the elliptic Cauchy problem under consideration. 

\end{abstract}

\maketitle


\section{Introduction} 

We focus our attention on the following Poisson-type problem:
\begin{equation}
    \begin{aligned}\label{mainprob}
    -\De u&=f\hspace{6mm}\text{ in }\Om,\\
    u&=u_0 \hspace{6mm}\text{on }S,\\
    \frac{\pa u}{\pa n}&=q_0\hspace{6.6mm}\text{on }S,
\end{aligned}
\end{equation} 
where $\Om\sub\R^{d}$ ($d\in \N$) is a bounded domain  with a piecewise  $C^1$  boundary $\pa \Om$. Here  $\pa \Om$ is decomposed into two mutually disjoint  (non-empty) parts $S$ and $L$. The part $S$ is referred to as being "accessible", while the part $L$ is called "inaccessible".  The special feature of problem \eqref{mainprob} is that on $S$ the boundary information is over-determined (in the sense that both Dirichlet and Neumann data are prescribed), while on $L$ the boundary data is missing (or unavailable due to an incomplete modeling, or simply inaccessible\footnote{Think, for instance,  at applications to materials science where measuring devices are unable to unveil the wanted information deep down inside the target material.}). Problems of this type are commonly referred to as elliptic Cauchy problems. 
Such scenarios occur in a variety of real-world applications, including  inverse heat conduction \cite{Marin}, inverse electrocardiography \cite{Eduardo}, electromagnetic problems governed by Maxwell's equations  \cite{Darbas}, linear elasticity \cite{Bastien},  and game theory \cite{Habbal}. Our own motivation stems from handling homogenization questions posed for perforated periodic media with defects; see e.g. \cite{Lions}. It is well-known that around the location of the defect it is not clear how to formulate boundary conditions for the corresponding cell problem, leaving open the possibility of treating the (defect)  situation as an elliptic Cauchy problem, with the "accessible" part $S$ being endowed with periodic boundary conditions, and the remaining part $L$ staying "inaccessible".  

  The problem \eqref{mainprob} is classically ill-posed in the sense of Hadamard. In fact, both the existence of a solution and its continuous dependence on the Cauchy data may fail. Nevertheless, whenever a solution exists, it is unique; see, e.g., \cite{leitao}. Regarding existence, even when $f=0$ and the Cauchy data are analytic functions, the Cauchy-Kovalevskaya theorem provides only local existence in a neighborhood of $S$ and it is general impossible to extend to the whole of $\Omega$. This inherent ill-posedness of the problem makes it necessary for us to assume from the outset that \eqref{mainprob} has a solution. Once this assumption is made, one of the main questions is:
  
  {\em Is it possible to construct stable and computable approximations to the solution of problem \eqref{mainprob}? }

There are many ways to address this question. What concerns our own investigations, we have chosen as starting point for the discussion the iterative scheme of Kozlov, Maz'ya, and Fomin cf. \cite{mazya_kozlov,mazya}.  This transforms the ill-posed Cauchy problem \eqref{mainprob} into a sequence of well-posed boundary-value problems. At each step, the unknown information on the boundary part $L$ is approximated using suitable information  from the previous iteration. This iterative perspective is closely related in spirit to the  Dirichlet--Neumann alternating method, which is widely used in domain decomposition exercises for the iterative exchange of complementary boundary information between the subdomains (see \cite{Emil} and references therein for more details). The approach in the spirit of \cite{mazya} is attractive since it improves numerical stability while allowing the unknown boundary data to be progressively updated (see e.g. \cite{leitao}, \cite{Engl}, \cite{Jourhmane}). Thus, each step is well-posed and computationally tractable, while the resulting sequence collects approximations of the solution to the original Cauchy problem.

The convergence of this iterative scheme to the solution of \eqref{mainprob} has been proven earlier (cf. e.g. \cite{mazya} and \cite{leitao}).
Our main contribution in this work is a novel convergence proof using the framework of alternating orthogonal projections, in particular von Neumann's alternating projection theorem (see \cite{vNeu1950} and \cite{badea}). More precisely, we construct a suitable Hilbert space $\mathcal{H}$ and identify two fixed closed subspaces of $\mathcal{H}$, independent of the iteration index, such that each step of the iterative scheme is represented precisely by successive orthogonal projections onto these two subspaces (see Theorem \ref{Th 3.2.1} -- our main result reported here). Consequently, the (strong) convergence of the iterative scheme can now be understood in terms of convergence of repeated orthogonal projections, exactly as in von Neumann's theorem. This is providing a unified functional-analytic interpretation of the proposed iterative scheme.

Once convergence is established, investigating the convergence rate is a natural next step. The theory surrounding alternating projections provides a general convergence rate for the iteration error. Without going into detail, denote by $e^{(n)}$ the error after $n$ iterations (or projections), then
\begin{equation}
    \label{intRate}
    \|e^{(n)}\|_\mathcal{H}\le\rho^n\|e^{(0)}\|_\mathcal{H}.
\end{equation}
The quantity $\rho\in[0,1]$ is expressed in terms of a geometrical quantity called the Friedrichs angle which depends on the aforementioned closed subspaces of $\mathcal{H}$ (and in particular, on $\Omega$). 
Unfortunately, in our case, this approach does not in general provide any rate of convergence. Indeed, we provide an explicit example of a domain $\Omega$ for which $\rho=1$ in (\ref{intRate}). We believe that this is due to the actual structure of problem \eqref{mainprob}, but the clarification of this matter requires further attention and goes beyond the purpose of this work.

On the other hand, we think it would be very interesting (and of clear practical relevance) to relate an eventual convergence rate of the iteration scheme with the measure $|L|$ of the unknown portion of the boundary. Intuitively, we expect that the larger the inaccessible part of the boundary is, the slower the iterations will converge. With these ideas in mind, we plan to revisit the convergence proof of the iterative scheme form a different point of view, hoping to be able to get some insight in the structure of the convergence rate. At present, we restrict ourselves by including a formal calculation suggesting that, under additional assumptions, the quantity $|L|$ may provide an upper bound for $\rho$ (see Remark \ref{formalRemark}).
 
The paper is organized as follows. After a brief introduction, we make our assumptions about the data, and  we introduce the functional setting. In Section 2, we discuss assumptions and notations. In Section 3, we describe the iterative scheme for which each step is well-posed. In Section 4, we analyze the convergence of this iterative scheme by means of the method of alternating orthogonal projections. In Section 5, we discuss the quantitative convergence rate of the iterative scheme via the Friedrichs angle between the two subspaces of $\mathcal{H}$. Section 6 contains our conclusions from this work. Finally, several auxiliary results and technical details used in the mathematical analysis of our approach are collected in the Appendix.

\section{Assumptions on data. Basic notation}
In this section, we collect the restrictions on the data and parameters that are required for our approach to work. We denote by  \ref{A_1}--\ref{A_3} the assumptions on the data and parameters of the original problem, and by \ref{B}  the additional assumptions required by the iterative scheme introduced in Section \ref{se 2.1}. 

Our starting point is precisely as in \cite[section~1.2]{leitao}: We consider the exact solution $u$ to our problem \eqref{mainprob} to be available and lying in $u\in H^1(\Om)$, i.e. $u$ solves the weak formulation of the Poisson equation posed in $\Om$ and also satisfies the Cauchy data.
To specify the functional setting we have in mind, we list now a the assumptions we need on data and parameters, viz.
\begin{enumerate}[label=$(A_{\arabic*})$]
\item\label{A_1} $S\neq\emptyset \text{ and } L\neq \emptyset$;
\item\label{A_2} $u_0\in H^\frac{1}{2}(S)$ and $q_0\in H^{-\frac{1}{2}}(S)$;
\item \label{A_3} $f\in L^2(\Om)$;
\end{enumerate}
\begin{enumerate}[label=$(B_{})$]
\item\label{B} $p^{(0)}\in H^{-\frac{1}{2}}(L)$.
\end{enumerate}


Assumptions \ref{A_1}--\ref{A_3} and \ref{B} are mild, fitting for many possible physical interpretations of \eqref{mainprob}. They correspond to the choice of the assumed regularity of the  exact solution.
\\
It is useful to define the subspace of $H^1(\Om)$ consisting of (Sobolev) functions that vanish on $S$, that is
$$H_{DS}^1(\Om):=\{\phi\in H^1(\Om)\hspace{2mm}\text{such that}\hspace{2mm} \phi|_{S}=0\}.$$
The norm on this space is the same as the standard $H^1$-norm
$$\|\phi\|_{H_{DS}^1(\Om)}:=\|\phi\|_{H^1(\Om)}.$$
Analogously, we define
$$
H_{DL}^1(\Om):=\{\phi\in H^1(\Om)\quad\text{such that}\quad\phi|_{L}=0\}.
$$
We refer the reader to \cite{Brezis} and \cite{Evans} for the relevant background on Sobolev spaces.

\section{Description of the iterative scheme and well-posedness}
In this section, we construct an iterative scheme based on the approach proposed in \cite{mazya}. We first introduce the iterative scheme and then establish its well-posedness.
\subsection{Iterative scheme}\label{se 2.1}
We now present the iterative scheme step by step. Essentially, the scheme consists of three stages: \texttt{Initial step},\texttt{ Odd steps}, and \texttt{Even steps}. Each step is designed to ensure that the overall scheme is well-posed.\\
The iterative scheme described in \eqref{IS}--\eqref{ES}, viz.
\begin{mybox}{\texttt{Initial step}}
\begin{equation}
    \begin{aligned}\label{IS}
    -\De u^{(0)}&=f \hspace{7.6mm}\text{in }\Om,\\
    u^{(0)}&=u_0\hspace{6mm}\text{on } S,\\
    \frac{\pa u^{(0)}}{\pa n}&=p^{(0)}\hspace{4.5mm}\text{on } L,
\end{aligned}
\end{equation}
where $p^{(0)}$ is our initial guess.
\end{mybox}
\begin{mybox}{\texttt{Odd steps}}
\begin{equation}
    \begin{aligned}\label{OS}
    -\De u^{(2k+1)}&=f \hspace{10mm}\text{in }\Om,\\
    u^{(2k+1)}&=u^{(2k)}\hspace{5mm}\text{on }L,\\
    \frac{\pa u^{(2k+1)}}{\pa n}&=q_0\hspace{9.5mm}\text{on } S.
\end{aligned}
\end{equation}
\end{mybox}
\begin{mybox}{\texttt{Even steps}}
\begin{equation}
    \begin{aligned}\label{ES}
    -\De u^{(2k+2)}&=f \hspace{18.7mm}\text{in}\hspace{3mm}\Om,\\
    u^{(2k+2)}&=u_0\hspace{17mm}\text{on}\hspace{3mm} S,\\
    \frac{\pa u^{(2k+2)}}{\pa n}&=\frac{\pa u^{(2k+1)}}{\pa n}\hspace{7mm}\text{on}\hspace{3mm} L.
\end{aligned}
\end{equation}
\end{mybox}
\subsection{Well-posedness}
Now, we establish the well-posedness of the iterative scheme \eqref{IS}--\eqref{ES}.
\begin{definition}\label{2.2.1}
    We say $u^{(0)}\in H^1(\Om)$ is a weak solution of problem \eqref{IS} if and only if the following identity holds:
    \begin{align*}
        \underset{\Om}{\int}\Gr u^{(0)}\Gr \phi\dx&=\underset{\Om}{\int} f\phi\dx +\underset{L}{\int}p^{(0)}\phi 
          \dsi,
    \end{align*}
    $\text{for all } \phi\in H_{DS}^1(\Om)$ with $u^{(0)}=u_0$ on $S$.
\end{definition}
\begin{definition}\label{2.2.2}
   For each $k\in\N\cup\{0\}$, we say $u^{(2k+1)}\in H^1(\Om)$ is a weak solution of problem \eqref{OS} if and only if the following identity holds:
    \begin{align*}
        \underset{\Om}{\int}\Gr u^{(2k+1)}\Gr \psi\dx&=\underset{\Om}{\int} f\psi\dx+\underset{S}{\int}q_0\psi\dsi ,
    \end{align*}
    $\text{for all } \psi\in H_{DL}^1(\Om)$ with $u^{(2k+1)}|_L=u^{(2k)}|_L$.
\end{definition}
\begin{definition}\label{2.2.3}
   For each $k\in\N\cup\{0\}$, we say $u^{(2k+2)}\in H^1(\Om)$ is a weak solution of problem \eqref{ES} if and only if the following identity holds:
    \begin{align*}
        \underset{\Om}{\int}\Gr u^{(2k+2)}\Gr \phi\dx&=\underset{\Om}{\int} f\phi\dx +\underset{L}{\int}\frac{\pa u^{(2k+1)}}{\pa n}\phi\dsi,
    \end{align*}
    $\text{for all }\phi\in H_{DS}^1(\Om)$ with $u^{(2k+2)}=u_0$ on $S$.
\end{definition}
\begin{theorem}
     Let Assumptions \ref{A_1}--\ref{A_3} and \ref{B} hold. For any $k\in\N\cup\{0\}$ arbitrarily fixed, Problem \eqref{IS}, Problem \eqref{OS}, and Problem \eqref{ES} are well-posed in the sense of Definition \ref{2.2.1}, Definition \ref{2.2.2}, and  Definition \ref{2.2.3}, respectively.
\end{theorem}
\begin{proof}
   Proposition \ref{Pr A.1} shows that the choice of boundary conditions at each iterative step makes sense.
    By applying the Lax-Milgram lemma (see Theorem 3.1.4 in \cite{kesavan}), we obtain the well-posedness of Problem \eqref{IS}--\eqref{ES}.
\end{proof}
\section{Convergence analysis}

In this section, we wish to describe the convergence of the iterative scheme proposed in the previous section. We describe several important properties that are essential for understanding how the iterative scheme works. We then prove that the approximating sequence generated by the iterative scheme converges to the exact solution of the main problem. 

It is worth noting that the convergence study of this iterative scheme has been investigated before by other authors. In the framework of this work, we wish to provide a new proof based on the method of alternating orthogonal projections (often associated with the name of John von Neumann).
\subsection{Auxiliary results on approximating errors}
By our assumption, $u$ is an exact solution to problem \eqref{mainprob}. Let $u^{(k)}$ be a weak solution of \eqref{IS}--\eqref{ES} $\text{for all } k\in \N\cup\{0\}$.\\ We define the approximating errors $e^{(k)}:=u^{(k)}-u$ $\text{for all } k\in \N\cup\{0\}$.\\
For all $k\in \N\cup\{0\}$, $e^{(k)}$ satisfy the following equations:
 \begin{equation}
     \begin{aligned}\label{ini.erro}
    -\De e^{(0)}&=0 \hspace{1.79cm}\text{in }\Om,\\
   e^{(0)} &=0\hspace{1.8cm}\text{on } S,\\
    \frac{\pa e^{(0)}}{\pa n}&=p^{(0)}-\frac{\pa u}{\pa n}\hspace{5mm}\text{on }L,
\end{aligned}
 \end{equation}

  \begin{equation}
       \begin{aligned}\label{odderro}
    -\De e^{(2k+1)}&=0 \hspace{10mm}\text{in }\Om,\\
    e^{(2k+1)}&=e^{(2k)}\hspace{5mm}\text{on } L,\\
    \frac{\pa e^{(2k+1)}}{\pa n}&=0\hspace{10.5mm}\text{on } S,
\end{aligned}
  \end{equation}
and
\begin{equation}
    \begin{aligned}\label{eve.erro}
    -\De e^{(2k+2)}&=0 \hspace{17.5mm}\text{in }\Om,\\
    e^{(2k+2)}&=0\hspace{17.8mm}\text{on }S,\\
    \frac{\pa e^{(2k+2)}}{\pa n}&=\frac{\pa e^{(2k+1)}}{\pa n}\hspace{6.5mm}\text{on } L.
\end{aligned} 
\end{equation} 
\begin{lemma}\label{Le 3.1.1}
    Let  \ref{A_1}--\ref{A_3} and \ref{B} hold. If  $e^{(k)}$ satisfies \eqref{ini.erro}--\eqref{eve.erro} $\text{for all}\, k\in\N\cup\{0\}$, then $\|\Gr e^{(k)}\|_{L^2(\Om)}$ is monotonically decreasing with respect to $k\in \N\cup \{0\}$.
\end{lemma}
\begin{proof}
    Let $\psi\in H_{DL}^1(\Om)$. From \eqref{odderro}, we get
\begin{align*}
    \underset{\Om}{\int} \Gr e^{(2k+1)}\Gr\psi\dx=0\hspace{1cm}\text{for all }\psi\in H_{DL}^1(\Om)
\end{align*}
with $e^{(2k+1)}=e^{(2k)}$ on $L$.\\
Take $\psi:=e^{(2k+1)}-e^{(2k)}\in H_{DL}^1(\Om)$.\\
Therefore, we obtain
$$\underset{\Om}{\int} \Gr e^{(2k+1)}\Gr(e^{(2k+1)}-e^{(2k)})\dx=0$$
which gives
$$\underset{\Om}{\int} |\Gr e^{(2k+1)}|^2\dx=\underset{\Om}{\int} \Gr e^{(2k+1)}\Gr e^{(2k)}\dx.$$
Applying the Cauchy-Schwarz inequality, we obtain that

    \begin{align*}
 \|\Gr e^{(2k+1)}\|_{L^2(\Om)}^2&\leq \|\Gr e^{(2k+1)}\|_{L^2(\Om)}\|\Gr e^{(2k)}\|_{L^2(\Om)}.
\end{align*}
This implies that 
\begin{equation}\label{8}
    \|\Gr e^{(2k+1)}\|_{L^2(\Om)}\leq\|\Gr e^{(2k)}\|_{L^2(\Om)}\hspace{4mm}\text{for all }k\in \N\cup\{0\}.
\end{equation}
From \eqref{eve.erro}, we get
\begin{equation}\label{9}
    \underset{\Om}{\int}\Gr e^{(2k+2)}\Gr\phi\dx=\underset{L}{\int}\frac{\pa e^{(2k+1)}}{\pa n}\phi\dsi\hspace{1cm}\text{for all }\phi\in H_{DS}^1(\Om).
\end{equation}
Applying Green's formula, we can write
\begin{align*}
    \underset{\Om}{\int}\Gr e^{(2k+1)}\Gr\xi\dx&=\underset{\pa\Om}{\int}\frac{\pa e^{(2k+1)}}{\pa n}\xi\dsi \hspace{1cm}\text{for all }\xi\in H^1(\Om)\\
    &=\underset{L}{\int}\frac{\pa e^{(2k+1)}}{\pa n}\xi\dsi\hspace{1cm}\text{for all }\xi\in H^1(\Om).
\end{align*}
Taking $\phi\in H_{DS}^1(\Om)$, we get
\begin{equation}\label{10}
    \underset{\Om}{\int}\Gr e^{(2k+1)}\Gr\phi\dx=\underset{L}{\int}\frac{\pa e^{(2k+1)}}{\pa n}\phi\dsi.
\end{equation}
Together, \eqref{9} and \eqref{10} imply that
\begin{align*}
    \underset{\Om}{\int}\Gr e^{(2k+2)}\Gr\phi\dx=\underset{\Om}{\int}\Gr e^{(2k+1)}\Gr\phi\dx.
\end{align*}
Take $\phi:=e^{(2k+2)}\in H_{DS}^1$ to yield
\begin{align*}
    \underset{\Om}{\int}|\Gr e^{(2k+2)}|^2\dx&=\underset{\Om}{\int}\Gr e^{(2k+1)}\Gr e^{(2k+2)}\dx.
\end{align*}
Applying the Cauchy-Schwarz inequality, we obtain that
    \begin{align*}
     \|\Gr e^{(2k+2)}\|_{L^2(\Om)}^2&\leq \|\Gr e^{(2k+1)}\|_{L^2(\Om)}\|\Gr e^{(2k+2)}\|_{L^2(\Om)}.
\end{align*}
This implies that 
\begin{equation}\label{11}
    \|\Gr e^{(2k+2)}\|_{L^2(\Om)}\leq\|\Gr e^{(2k+1)}\|_{L^2(\Om)}\hspace{4mm}\text{for all }k\in \N\cup\{0\}.
\end{equation}
By \eqref{8} and \eqref{11} we see that
$$0\leq \cdots\cdots\leq\|\Gr e^{(k+1)}\|_{L^2(\Om)}\leq\|\Gr e^{(k)}\|_{L^2(\Om)}\leq\cdots\cdots\leq\|\Gr e^{(0)}\|_{L^2(\Om)}.$$
\end{proof}
\begin{lemma}\label{Le 3.1.2}
  Let \ref{A_1}--\ref{A_3} and \ref{B} hold. If $e^{(k)}$ satisfies \eqref{ini.erro}--\eqref{eve.erro} for every $k\in \N\cup \{0\}$, then 
\begin{enumerate}
      \item $\Gr e^{(k+1)}$ is orthogonal to $\Gr(e^{(k)}-e^{(k+1)})$ for all $k\in\N\cup \{0\}$.
      \item $\|\Gr (e^{(k)}-e^{(k+1)})\|_{L^2(\Om)}\to 0$ as $k \to \infty$.
  \end{enumerate}
\end{lemma}
\begin{proof}[Proof (1)]
The following identity holds for all $k\in\N \cup \{0\}$:
    \begin{align*}
    &\|\Gr e^{(2k)}\|_{L^2(\Om)}^2- \|\Gr e^{(2k+1)}\|_{L^2(\Om)}^2\\&= \|\Gr e^{(2k)}\|_{L^2(\Om)}^2- \|\Gr e^{(2k)}-\Gr(e^{(2k)}-e^{(2k+1)})\|_{L^2(\Om)}^2\\
    &= \|\Gr e^{(2k)}\|_{L^2(\Om)}^2- \|\Gr e^{(2k)}\|_{L^2(\Om)}^2+2\underset{\Om}{\int}\Gr e^{(2k)}\Gr (e^{(2k)}-e^{(2k+1)})\dx\\
    &\hspace{4mm}- \|\Gr (e^{(2k)}-e^{(2k+1)})\|_{L^2(\Om)}^2\\
    &=2\underset{\Om}{\int}\Gr e^{(2k)}\Gr (e^{(2k)}-e^{(2k+1)})\dx- \|\Gr (e^{(2k)}-e^{(2k+1)})\|_{L^2(\Om)}^2.
\end{align*}
This implies that
\begin{equation}
   \begin{aligned}\label{12}
       & \|\Gr e^{(2k)}\|_{L^2(\Om)}^2- \|\Gr e^{(2k+1)}\|_{L^2(\Om)}^2\\
       &=2\underset{\Om}{\int}\Gr e^{(2k)}\Gr (e^{(2k)}-e^{(2k+1)})\dx- \|\Gr (e^{(2k)}-e^{(2k+1)})\|_{L^2(\Om)}^2.
   \end{aligned}
\end{equation}
Reformulating the last term of \eqref{12} by applying Green's formula, and using the facts that  $e^{(2k)}-e^{(2k+1)}=0$ on $L$ and $\frac{\pa e^{(2k+1)}}{\pa n}=0$ on $S$, we obtain
\begin{align*}
    &2\underset{\Om}{\int}\Gr e^{(2k)}\Gr (e^{(2k)}-e^{(2k+1)})\dx- \|\Gr (e^{(2k)}-e^{(2k+1)})\|_{L^2(\Om)}^2\\
    &=2\underset{S}{\int}\frac{\pa e^{(2k)}}{\pa n}(e^{(2k)}-e^{(2k+1)})\dsi-\|\Gr (e^{(2k)}-e^{(2k+1)})\|_{L^2(\Om)}^2\\
    &=2\underset{S}{\int}\frac{\pa}{\pa n}(e^{(2k)}-e^{(2k+1)})(e^{(2k)}-e^{(2k+1)})\dsi-\|\Gr (e^{(2k)}-e^{(2k+1)})\|_{L^2(\Om)}^2\\
    &=2\underset{\Om}{\int}|\Gr(e^{(2k)}-e^{(2k+1)})|^2\dx-\|\Gr (e^{(2k)}-e^{(2k+1)})\|_{L^2(\Om)}^2\\
    &=\|\Gr (e^{(2k)}-e^{(2k+1)})\|_{L^2(\Om)}^2.
\end{align*}
Hence, for all $k\in \N\cup\{0\}$, we obtain the first key identity
\begin{equation}\label{13}
    \|\Gr e^{(2k)}\|_{L^2(\Om)}^2- \|\Gr e^{(2k+1)}\|_{L^2(\Om)}^2=\|\Gr (e^{(2k)}-e^{(2k+1)})\|_{L^2(\Om)}^2.
\end{equation}
Proceeding similarly, for all $k\in \N\cup\{0\}$ we have
\begin{align*}
   &\|\Gr e^{(2k+1)}\|_{L^2(\Om)}^2- \|\Gr e^{(2k+2)}\|_{L^2(\Om)}^2\\
    &= \|\Gr e^{(2k+1)}\|_{L^2(\Om)}^2- \|\Gr e^{(2k+1)}-\Gr(e^{(2k+1)}-e^{(2k+2)})\|_{L^2(\Om)}^2\\
    &= \|\Gr e^{(2k+1)}\|_{L^2(\Om)}^2- \|\Gr e^{(2k+1)}\|_{L^2(\Om)}^2+2\underset{\Om}{\int}\Gr e^{(2k+1)}\Gr (e^{(2k+1)}-e^{(2k+2)})\dx\\
    &\hspace{4mm}- \|\Gr (e^{(2k+1)}-e^{(2k+2)})\|_{L^2(\Om)}^2\\
    &=2\underset{\Om}{\int}\Gr e^{(2k+1)}\Gr (e^{(2k+1)}-e^{(2k+2)})\dx- \|\Gr (e^{(2k+1)}-e^{(2k+2)})\|_{L^2(\Om)}^2,
\end{align*}
since $\frac{\pa}{\pa n} (e^{(2k+1)}-e^{(2k+2)})=0$ on $L$, and $e^{(2k+2)}=0$ on $S$, we consequently obtain 
\begin{align*}
     & 2\underset{\Om}{\int}\Gr e^{(2k+1)}\Gr (e^{(2k+1)}-e^{(2k+2)})\dx- \|\Gr (e^{(2k+1)}-e^{(2k+2)})\|_{L^2(\Om)}^2\\
     &=2\underset{S}{\int}\frac{\pa}{\pa n}(e^{(2k+1)}-e^{(2k+2)}) e^{(2k+1)}\dsi - \|\Gr (e^{(2k+1)}-e^{(2k+2)})\|_{L^2(\Om)}^2\\
     &=2\underset{S}{\int}\frac{\pa}{\pa n}(e^{(2k+1)}-e^{(2k+2)})( e^{(2k+1)}-e^{2k+2)})\dsi - \|\Gr (e^{(2k+1)}-e^{(2k+2)})\|_{L^2(\Om)}^2\\
     &=2\underset{\Om}{\int}|\Gr (e^{(2k+1)}-e^{(2k+2)})|^2\dx- \|\Gr (e^{(2k+1)}-e^{(2k+2)})\|_{L^2(\Om)}^2\\
     &=\|\Gr (e^{(2k+1)}-e^{(2k+2)})\|_{L^2(\Om)}^2.
\end{align*}
Hence, for all $k\in \N\cup\{0\}$, we obtain the second key identity
\begin{equation}\label{14}
    \|\Gr e^{(2k+1)}\|_{L^2(\Om)}^2- \|\Gr e^{(2k+2)}\|_{L^2(\Om)}^2=\|\Gr (e^{(2k+1)}-e^{(2k+2)})\|_{L^2(\Om)}^2.
\end{equation}
Together, \eqref{13} and \eqref{14} imply that
\begin{align*}
    \|\Gr e^{(k)}\|_{L^2(\Om)}^2- \|\Gr e^{(k+1)}\|_{L^2(\Om)}^2=\|\Gr (e^{(k)}-e^{(k+1)})\|_{L^2(\Om)}^2 \hspace{4mm}\text{ for all } k\in\N\cup\{0\}.
\end{align*}
Hence, $\Gr e^{(k+1)}$ is orthogonal to $\Gr(e^{(k)}-e^{(k+1)})$ $\text{for all } k\in \N\cup \{0\}$.
\end{proof}
\begin{proof}[Proof (2)]
    By Lemma \ref{Le 3.1.1} and  Lemma \ref{Le 3.1.2} $(1)$, we obtain that 
    $$\|\Gr (e^{(k)}-e^{(k+1)})\|_{L^2(\Om)}\to 0\hspace{7mm}\text{as}\hspace{2mm}k\to\infty.$$
\end{proof}
We now connect the proposed iterative scheme with the method of alternating orthogonal projections. In this context, we present the following useful proposition, originally proved by J. von Neumann \cite[Chapter 13, Theorem 13.7]{vNeu1950}.
\begin{proposition}\label{Pr 3.1.3}
    Let $\mathcal{H}$ be a Hilbert space, and let $V_D$ and $V_N$ be closed subspaces of $\mathcal{H}$. Denote by $P_{V_D}$ and $P_{V_N}$ the orthogonal projections of $\mathcal{H}$ onto $V_D$ and $V_N$, respectively. Then, for every $x\in\mathcal{H}$,
    $$\underset{n\to\infty}{\text{lim}}\|(P_{V_D}\circ P_{V_N})^n(x)-P_{V_D\cap V_N}(x)\|_{\mathcal{H}}=0,$$
    where $P_{V_D\cap V_N}$ is the orthogonal projection of $\mathcal{H}$ onto the intersection $V_D\cap V_N$. 
\end{proposition}

\subsection{Main result: Convergence of the iterative scheme}\label{sec 3.2}
In this section, we prove that the approximating sequence generated by the iterative scheme converges to the solution of the original problem \eqref{mainprob} by applying Proposition \ref{Pr 3.1.3}. To this end, we introduce a suitable Hilbert space and two closed subspaces in this Hilbert space, which allow us to apply Proposition \ref{Pr 3.1.3}. We introduce these spaces as follows.\\
We define the quotient space  $$\mathcal{H}:=H^1(\Om)/ \R.$$
The space $\mathcal{H}$ is a Hilbert space equipped with the inner product $$\langle v,w\rangle_\mathcal{H}:=\underset{\Om}{\int}\Gr v \Gr w \dx.$$ 
Throughout, elements of $\mathcal{H}$ are understood as equivalence classes modulo constants. Thus, whenever $m\in\mathcal{H}$, statements about $m$ are interpreted in the sense of its equivalence class.

We introduce the following two subspaces of $\mathcal{H}$, namely
\begin{equation}
    \label{VdDef}
    V_D:=\{v\in \mathcal{H}: -\De v=0\hspace{1mm}\text{weakly in}\hspace{1mm}\Om \hspace{1mm}\text{and}\hspace{1mm}  v=0 \hspace{1mm}\text{on}\hspace{1mm} S\}\footnote{We mean that every element of $V_D$ admits a unique representative in its equivalence class that vanishes on $S$. }
\end{equation}
and
\begin{equation}
    \label{VnDef}
    V_N:=\{v\in \mathcal{H}: -\De v=0\hspace{1mm}\text{weakly in}\hspace{1mm}\Om\hspace{1mm}\text{and}\hspace{1mm} \frac{\pa v}{\pa n}=0 \hspace{1mm}\text{on}\hspace{1mm} S\}.
\end{equation}
According to Proposition \ref{Pr A.2}, both $V_D$ and $V_N$ are closed subspaces of $\mathcal{H}$. Furthermore, by \ref{Th A.1}
\begin{equation}
     \begin{aligned}
     \label{VdVnIntersect}
    V_D\cap V_N&=\{v\in \mathcal{H}: -\De v=0\hspace{1mm}\text{weakly in}\hspace{1mm}\Om,v=0\hspace{1mm}\text{on}\hspace{1mm} S,\hspace{1mm}\text{and}\hspace{1mm}\frac{\pa v}{\pa n}=0 \hspace{1mm}\text{on}\hspace{1mm} S\}\\
    &=\{0\}.
\end{aligned}
\end{equation}

\begin{theorem}\label{Th 3.2.1}
  Assume \ref{A_1}--\ref{A_3} and \ref{B} hold. Let $u$ be an exact solution of problem \eqref{mainprob} and $u^{(k)}$ be a weak solution of \eqref{IS}--\eqref{ES} $\text{for all } k\in \N\cup\{0\}$. Then  $u^{(k)} \to u$ strongly in $H^1(\Om)$ as $k\to \infty$.
\end{theorem}
\begin{proof}
   Recall that $e^{(k)}=u^{(k)}-u$ and that $e^{(k)}$ satisfies \eqref{ini.erro}--\eqref{eve.erro} $\text{for all } k\in \N\cup \{0\}$.\\
We work with the Hilbert space $\mathcal{H}$ and the two closed subspaces $V_D$ and $V_N$ introduced in Section \ref{sec 3.2}.\\
Let $$P_{V_N}:\mathcal{H}\to V_N$$ be the orthogonal projection onto $V_N$.\\
We claim that $$e^{(2k+1)}=P_{V_N}(e^{(2k)})\hspace{4mm}\text{for all }k\in \N\cup\{0\}.$$
Let $\xi\in V_N$. For all $k\in \N\cup\{0\}$, we can write
\begin{align*}
    \underset{\Om}{\int}\Gr(e^{(2k)}-e^{(2k+1)})\Gr\xi\dx&=\underset{\pa \Om}{\int}(e^{(2k)}-e^{(2k+1)})\frac{\pa\xi}{\pa n}\dsi\\
    &=\underset{L}{\int}(e^{(2k)}-e^{(2k+1)})\frac{\pa\xi}{\pa n}\dsi\hspace{8mm} \text{for all }\xi\in V_N.
\end{align*}
Since $ e^{(2k+1)}=e^{(2k)}$ on $L$, we obtain
\begin{equation}\label{15}
    \underset{\Om}{\int}\Gr(e^{(2k)}-e^{(2k+1)})\Gr\xi\dx=0\hspace{8mm} \text{for all }\xi\in V_N.
    \end{equation}
From \eqref{odderro}, we see that $e^{(2k+1)}\in V_N$.
    By \eqref{15} and the fact that $e^{(2k+1)}\in V_N$, we get 
    \begin{equation}\label{16}
        e^{(2k+1)}=P_{V_N}(e^{(2k)})\hspace{4mm}\text{for all }k\in \N\cup\{0\}.
    \end{equation}
    Let $$P_{V_D}:\mathcal{H}\to V_D$$ be the orthogonal projection onto $V_D$.\\
Proceeding similarly, we claim that 
$$e^{(2k+2)}=P_{V_D}(e^{(2k+1)})\hspace{4mm}\text{for all }k\in \N\cup\{0\}.$$
Let $\chi\in V_D$. For all $k\in \N\cup\{0\}$, we can write
\begin{align*}
    \underset{\Om}{\int}\Gr(e^{(2k+1)}-e^{(2k+2)})\Gr\chi\dx&=\underset{\pa \Om}{\int}\frac{\pa}{\pa n}(e^{(2k+1)}-e^{(2k+2)})\chi\dsi\\
    &=\underset{L}{\int}\frac{\pa}{\pa n}(e^{(2k+1)}-e^{(2k+2)})\chi\dsi\hspace{8mm} \text{for all }\chi\in V_D.
\end{align*}
Since $\frac{\pa e^{(2k+2)}}{\pa n}=\frac{\pa e^{(2k+1)}}{\pa n}$ on $L$, we obtain
\begin{equation}
    \begin{aligned}\label{17}
    \underset{\Om}{\int}\Gr(e^{(2k+1)}-e^{(2k+2)})\Gr\chi\dx&=0\hspace{8mm}\text{for all }\chi\in V_D.
\end{aligned}
\end{equation}
  From  \eqref{eve.erro}, we have $e^{(2k+2)}\in V_D$. By \eqref{17} and the fact that  $e^{(2k+2)}\in V_D$,we get 
  \begin{equation}\label{18}
      e^{(2k+2)}=P_{V_D}(e^{(2k+1)})\hspace{4mm}\text{for all }k\in \N\cup\{0\}.
  \end{equation}
 Finally, we can combine \eqref{16} and \eqref{18} to get
\begin{align*}
  e^{(2k+2)}&=P_{V_D}\circ P_{V_N}(e^{(2k)}) \\
  &=(P_{V_D}\circ P_{V_N})^2(e^{(2k-2)})\\
  &\hspace{1cm}\vdots\\
  &= (P_{V_D}\circ P_{V_N})^{k+1}(e^{(0)}).
\end{align*}
By Proposition \ref{Pr 3.1.3}, we obtain
\begin{equation}\label{19}
    (P_{V_D}\circ P_{V_N})^{k+1}(e^{(0)})\to P_{V_D\cap V_N}(e^{(0)})\hspace{6mm}\text{as}\hspace{1.5mm} k\to\infty.
\end{equation}
By \eqref{19} and \eqref{VdVnIntersect}, we deduce that
$$
    \lim_{k\rightarrow\infty}(P_{V_D}\circ P_{V_N})^{k+1}(e^{(0)})\to P_{V_D\cap V_N}(e^{(0)})=P_{\{0\}}(e^{(0)})=0
$$
in $\mathcal{H}$. Consequently, we also get
\begin{equation}\label{21}
    e^{(2k+2)}\to 0\hspace{2mm}\text{strongly in}\hspace{1mm}\mathcal{H} \hspace{7mm}\text{as}\hspace{1.5mm} k\to\infty.
\end{equation} 
Since $e^{(2k+1)}=P_{V_N}(e^{(2k)})$ and $e^{(2k+2)}\to 0$ strongly in $\mathcal{H}$ as $k\to\infty$, this implies that
\begin{equation}\label{22}
    e^{(2k+1)}\to 0\hspace{2mm}\text{strongly in}\hspace{1mm}\mathcal{H}\hspace{7mm}\text{as}\hspace{1.5mm} k\to\infty.
\end{equation}
By \eqref{21} and \eqref{22}, we get
\begin{equation}\label{23}
    e^{(k)}\to 0 \hspace{2mm}\text{strongly in}\hspace{1mm}\mathcal{H}\hspace{7mm}\text{as}\hspace{1.5mm} k\to\infty.
\end{equation}
It is not sufficient to show that $e^{(k)}\to 0$ strongly in $\mathcal{H}$ as $k\to 0$. We also need to establish that  $e^{(k)}\to 0$ strongly in $H^1(\Om)$ as $k\to\infty$.\\
Consequently, \eqref{23} leads to $$\|\Gr e^{(k)}\|_{L^2(\Om)}\to 0\hspace{7mm}\text{as}\hspace{1.5mm} k\to\infty.$$
Since $e^{(2k+2)}=0$ on $S$,  the Poincar\'e inequality implies that 
\begin{equation}\label{24}
    e^{(2k+2)}\to 0 \hspace{3mm} \text{strongly in}\hspace{2mm} H^1(\Om)\hspace{7mm}\text{as}\hspace{1.5mm} k\to\infty.
\end{equation}
We denote $\Tilde{e}^{(2k+1)}:=e^{(2k+1)}-e^{(2k)}$ with $\Tilde{e}^{(2k+1)}\in H^1(\Om)$. By \eqref{odderro}, we get
$$\underset{\Om}{\int}\Gr(\Tilde{e}^{(2k+1)}+e^{(2k)})\Gr\psi\dx=0\hspace{1cm}\text{for all } \psi\in H_{DL}^1(\Om).$$
Taking $\psi=\Tilde{e}^{(2k+1)}\in H_{DL}^1$ in the last expression, we obtain
\begin{align*}
    \underset{\Om}{\int}\Gr(\Tilde{e}^{(2k+1)}+e^{(2k)})\Gr\Tilde{e}^{(2k+1)}\dx&=0.
    \end{align*}
    Rearranging the terms and applying the Cauchy-Schwarz inequality, we get
    \begin{align*}
    \underset{\Om}{\int}|\Gr\Tilde{e}^{(2k+1)}|^2\dx&\leq|\underset{\Om}{\int}\Gr e^{(2k)}\Gr\Tilde{e}^{(2k+1)}|\dx\\
    \|\Gr\Tilde{e}^{(2k+1)}\|_{L^2(\Om)}^2&\leq \|\Gr e^{(2k)}\|_{L^2(\Om)} \|\Gr\Tilde{e}^{(2k+1)}\|_{L^2(\Om)}.
\end{align*}
This implies that
\begin{equation}\label{25}
    \|\Gr\Tilde{e}^{(2k+1)}\|_{L^2(\Om)}\leq\|\Gr e^{(2k)}\|_{L^2(\Om)}.
\end{equation}
Applying the Poincar\'e inequality on $\Tilde{e}^{(2k+1)}$ and combining it with \eqref{24} and \eqref{25}, we get
\begin{equation}\label{26}
    \Tilde{e}^{(2k+1)}\to 0 \hspace{3mm} \text{strongly in}\hspace{2mm} H^1(\Om)\hspace{7mm}\text{as}\hspace{1.5mm} k\to\infty.
\end{equation} 
By \eqref{24} and \eqref{26}, and by setting $\Tilde{e}^{(2k+1)}=e^{(2k+1)}-e^{(2k)}$,we deduce that 
\begin{equation}\label{27}
    e^{(2k+1)}\to 0 \hspace{3mm} \text{strongly in}\hspace{2mm} H^1(\Om) \hspace{7mm}\text{as}\hspace{1.5mm} k\to\infty.
\end{equation}
By \eqref{24} and \eqref{27}, we finally see $$e^{(k)}\to 0 \hspace{2mm}\text{strongly in}\hspace{1mm}H^1(\Om)\hspace{7mm}\text{as}\hspace{1.5mm} k\to\infty,$$
which equivalently gives $$u^{(k)}\to u\hspace{2mm}\text{strongly in}\hspace{1mm}H^1(\Om)\hspace{7mm}\text{as}\hspace{1.5mm} k\to\infty.$$ This concludes the proof of our main result.
\end{proof} 

\section{In search of convergence rates}\label{se 4}
For practical application of the described iterative method, it is relevant to know what convergence rate is to be expected. We will address this matter using the concept of \emph{Friedrichs angle} (or the \emph{Friedrichs cosine}). Although this is a general quantity that can be computed for any pair of closed subspaces of a Hilbert space, we shall give the definition only in the special case of $V_D,V_N\subset\mathcal{H}$. Recall that $V_D, V_N$ are defined by (\ref{VdDef}) and (\ref{VnDef}). Denote by 
$$
A:=\left\{(u,v): u\in V_D \text{ with } \|\Gr u\|_{L^2(\Om)}\leq1\text{ and } v\in V_N \text{ with } \|\Gr v\|_{L^2(\Om)}\leq1\right\}.
$$
The Friedrichs angle of $V_D,V_N$ is, in our context, given by
\begin{equation}
    \nonumber
    c(V_D,V_N):=\sup_{(u,v)\in A}\left|\int_\Omega\Gr u\Gr v\dx\right|.
\end{equation}


 
Interestingly, the theorem of von Neumann on the convergence of alternating orthogonal projections also admits the following quantitative estimate (see \cite[Theorem 1.2]{Heinz} or Aronszajn  \cite[p.~378--379]{Aronszajn}):
\begin{equation}\label{errorbound}
    \|(P_{V_D}\circ P_{V_N})^n(x)-P_{V_D\cap V_N}(x)\|_{\mathcal{H}}\leq c(V_D,V_N)^{2n-1}\|x\|_{\mathcal{H}}\hspace{1cm}\text{for } x\in\mathcal{H}\text{ and }n\in\N.
\end{equation}

Consequently,
\begin{equation}
    \label{rateError}
    \|e^{(2n+2)}\|_\mathcal{H}\le c(V_D,V_N)^{2n+1}\|e^{(0)}\|_\mathcal{H}.
\end{equation}
Of course, if $c(V_D,V_N)<1$, then (\ref{rateError}) provides a geometrical convergence rate of the iteration error. Unfortunately, it is easy to give examples of $\Omega$ where $c(V_D,V_N)=1$ (see Proposition \ref{exampleProp} below). Therefore, in general, no convergence rate of $e^{(n)}$ is to be expected. 

Elaborating slightly, it can be shown \cite[Theorem~2]{weinert} that
\begin{equation}
    \label{operatorNorm}
    \|(P_{V_D}\circ P_{V_N})^n-P_{V_D\cap V_N}\|_{\textrm{Op}}=c(V_D,V_N)^{2n-1}\quad(n\in\mathbb{N}),
\end{equation}
where $\|T\|_{\textrm Op}=\sup\|Th\|_\mathcal{H}/\|h\|_\mathcal{H}$. Therefore, for each $n\in\mathbb{N}$, there is an initial error $e^{(0)}$ such that $\|e^{(2n+2)}\|_\mathcal{H}=c(V_D,V_N)^{2n+1}\|e^{(0)}\|_\mathcal{H}$. This together with the fact that we may have $c(V_D,V_N)=1$ shows that it is possible that for each $n\in\mathbb{N}$ there holds 
$$
\sup_{e^{(0)}\in\mathcal{H}}\frac{\|e^{(2n+2)}\|_\mathcal{H}}{\|e^{(0)}\|_\mathcal{H}}=1.
$$
In other words, there is an $\Omega$ with the property that for any $n\in\mathbb{N}$, there is an initial error $e^{(0)}$ such that the $(2n+2)$-th error has \emph{no improvement} compared to the initial error.
Formulated differently, given any $\epsilon>0$ and any $n\in\mathbb{N}$, there

This appears to be an obstruction inherent to the alternating projection theorem applied to our specific problem.

\begin{proposition}
\label{exampleProp}
   There exists a bounded domain $\Om\subset \R^2$  such that
   $$ c(V_D,V_N)=1,$$ where $V_D$ and $V_N$ are defined as in Section \ref{sec 3.2}.
\end{proposition}
\begin{proof}
The sum $V_D+V_N$ is closed if and only if $c(V_D,V_N)<1$   (see \cite[Lemma 1.3]{Heinz}). Thus, it suffices to give an explicit example of $\Omega$ and $S$ such that the corresponding space $V_D+V_N$ is not closed. To this end, 
let $\Om:=[0,1]\times[0,1]$, $S:=\{(x,0):0<x<1\}$, and $L:=\pa\Om\setminus S$.
It is known (see \cite[Theorem 1.1]{Zhe}) that $V_D+V_N$ is closed if and only if there exist a constant $K>0$ such that 
    \begin{equation}\label{closedinequality}
        \|u\|_{\mathcal{H}}\leq K\|u+v\|_{\mathcal{H}}\hspace{4mm}\text{for all } u\in V_D, v\in V_N.
    \end{equation}
Assume for a contradiction that \eqref{closedinequality} holds. Consider the sequences
 $(u_n)\subset V_D$  and $(v_n)\subset V_N$  given by $u_n:=\sin{(nx)}\sinh{(ny)}$ and $v_n:=-\sin{(nx)}\cosh{(ny)}$. Then
 $$\De u_n=-n^2\sin{(nx)}\sinh{(ny)}+n^2\sin{(nx)}\sinh{(ny)}=0$$
 and $$\De v_n=n^2\sin{(nx)}\cosh{(ny)}-n^2\sin{(nx)}\cosh{(ny)}=0,$$
 whence $u_n$ and $v_n$ are harmonic. Clearly, $u_n|_{S}=0$ and 
 $$\frac{\pa v_n}{\pa n}\Big|_{S}=-\frac{\pa v_n}{\pa y}\Big|_{S}=n\sin{(nx)}\sinh{(ny)}\big|_S=0.$$
 Therefore $u_n\in V_D$ and $v_n\in V_N$.
 Note that for all $n\in\mathbb{N}$
 \begin{align*}
     \|\Gr u_n\|_{L^2(\Om)}^2&=n^2\underset{0}{\int}^1\underset{0}{\int}^1\big(\cos^2{(nx)}\sinh^2{(ny)}+\sin^2{(nx)}\cosh^2{(ny)}\big)\dx\dy\\
     &=\frac{n}{4}\big(\sinh{(2n)}-\sin{(2n)}\big)= \frac{n}{8} e^{2n}(1-e^{-4n}-2\sin{(2n)}e^{-2n})\\
     &\ge\frac{ne^{2n}}{16},
 \end{align*}
 and
 \begin{align*}
     \|\Gr(u_n+v_n)\|_{L^2(\Om)}^2&=n^2\underset{0}{\int}^1\underset{0}{\int}^1\big(\sin^2(nx)+\cos^2(nx)\big) e^{-2ny}\dx\dy=\frac{n}{2}(1-e^{-2n})\\
     &\le\frac{n}{4}.
 \end{align*}
 Consequently,
 \begin{align}
     \nonumber
     \frac{ \|\Gr u_n\|_{L^2(\Om)}^2}{ \|\Gr(u_n+v_n)\|_{L^2(\Om)}^2}\ge\frac{e^{2n}}{4}.
 \end{align}
On the other hand, by \eqref{closedinequality}, there exists $K>0$ such that
\begin{align*}
    \frac{ \|\Gr u_n\|_{L^2(\Om)}^2}{ \|\Gr(u_n+v_n)\|_{L^2(\Om)}^2}\leq K^2.
\end{align*}
Whence,
\begin{align*}
    e^{2n}\leq 4K^2,
\end{align*}
which is a contradiction. Hence $V_D+V_N$ is not closed.
\end{proof}
\begin{remark}
\label{formalRemark}
    Interestingly, if we allow ourselves to 
argue formally, we notice the proportionality between $c(V_D,V_N)$ and the measure of the inaccessible part of the boundary $L$.  
We let \ref{A_1} hold, and consider that $V_D$ and $V_N$  to be  the same closed subspaces of $\mathcal{H}$ defined in Section \ref{sec 3.2} and, additionally, assumes that there exists a constant $C'>0$ such that  
 \begin{equation}\label{conditionbound}
     \left\|\frac{\pa v}{\pa n}\right\|_{L^\infty(L)}\leq C'\hspace{4mm}\text{ for every } v\in V_N.
 \end{equation}
Using the Green's formula, the fact that $u=0$ on $S$, and the standard trace-Poincar\'{e} machinery, we obtain
\begin{align}
    \nonumber
    \left|\int_\Omega\Gr u\Gr v\dx\right|&=\left|\int_L\frac{\partial v}{\partial n}u\dsi\right|\le C'\int_L|u|\dsi\le C'\sqrt{|L|}\|u\|_{L^2(\partial\Omega)}\\
    \nonumber
    &\le C'C'_\Omega\sqrt{|L|}\|\Gr u\|_{L^2(\Omega)}\le C_\Omega\sqrt{|L|}\|\Gr u\|_{L^2(\Omega)},
\end{align}
where $|L|$ denotes the surface measure of $L$. Hence,
\begin{equation}\label{conditionrate}
    c(V_D,V_N)\leq C_\Om\sqrt{|L|}. 
\end{equation}
In \eqref{conditionrate}, $C_\Om>0$ is a constant that depends on the set $\Om$ (via the trace inequality and Poincar\'{e}'s inequality), as well as the bound (\ref{conditionbound}).
This estimate indicates that, should it exists, the convergence rate is affected   by the size of the unknown part of the boundary $L$. In particular, $c(V_D,V_N)$ decreases as the measure of $L$ becomes smaller. Therefore, a smaller interface measure of $L$ leads to a smaller value of $c(V_D,V_N)$, which ultimately results in faster convergence, i.e. the more we know, the faster we expect to converge\footnote{This is (somewhat trivially) a rigorous statement in the one-dimensional case. In higher space dimensions, an assumption like \eqref{conditionbound} can only hold if one is able to lift sufficiently high the regularity of the target solution (perhaps by taking solutions from $H^3(\Om)$ instead of $H^1(\Om)$). However, this regularity lift would change the underlying Hilbert space $\mathcal{H}$ as well as the meaning of the spaces $V_D$ and $V_N$; so, the  setting of the alternating orthogonal projections method changes drastically and a new analysis would need to be done. }.
\end{remark}

\section{Conclusion}

In this work, we analyzed the geometric structure of an iterative scheme for a linear elliptic Cauchy problem. More precisely, by reformulating the scheme in terms of alternating orthogonal projections onto two fixed closed subspaces of a suitable Hilbert space, we characterize both the convergence of the iterative scheme and its limit. 

Our results provide a geometric framework for understanding and analyzing iterative methods for linear elliptic Cauchy problems. An important direction for future research is to investigate whether the present framework can be extended to more general classes of coupled elliptic Cauchy systems in the multiscale setting. Such an approach offers a potentially useful framework for solving inverse problems in real-world applications involving heterogeneous media.
\appendix
\renewcommand{\thesection}{\Alph{section}} 
\makeatletter
\let\oldthetheorem\thetheorem
\let\oldthelemma\thelemma
\let\oldtheproposition\theproposition
\let\oldthecorollary\thecorollary
\let\oldtheremark\theremark
\renewcommand{\thetheorem}{\thesection.\arabic{theorem}}
\renewcommand{\thelemma}{\thesection.\arabic{lemma}}
\renewcommand{\theproposition}{\thesection.\arabic{proposition}}
\renewcommand{\thecorollary}{\thesection.\arabic{corollary}}
\renewcommand{\theremark}{\thesection.\arabic{remark}}
\renewcommand{\thedefinition}{\thesection.\arabic{definition}}
\makeatother
\refstepcounter{section} 
\section*{Appendix \thesection}
We include in this appendix a few elementary auxiliary results that are useful throughout this work.
\begin{proposition}\label{Pr A.1}
    Let \ref{A_1} hold, and  $u\in H^1(\Om)$ satisfy $$-\De u=f\hspace{8mm}\text{in}\hspace{2mm}\Om,$$ where $f\in L^2(\Om)$. If $\frac{\pa u}{\pa n}\big|_L\in H^{-\frac{1}{2}}(L)$ then $\frac{\pa u}{\pa n}\big|_S\in H^{-\frac{1}{2}}(S)$.
\end{proposition}
\begin{proof}
   The following arguments are inspired by the proof of Proposition~1 in \cite{mazya}.\\
    Let $\varphi\in H^\frac{1}{2}(S)$ and define the functional $\mathcal{F}:H^\frac{1}{2}(S)\to \R$ via
     $$\mathcal{F}(\varphi):=\underset{\Om}{\int}(\Gr u\Gr \Tilde{\varphi}-f\Tilde{\varphi})\dx-\underset{L}{\int}\frac{\pa u}{\pa n}\Tilde{\varphi}\dsi,$$
    where $\Tilde{\varphi}\in H^1(\Om)$ with $\Tilde{\varphi}|_S=\varphi$. If $\Tilde{\varphi}_1, \Tilde{\varphi}_2\in H^1(\Om)$ with  $\Tilde{\varphi}_1=\Tilde{\varphi}_2$ on $S$, then
     $$\mathcal{F}((\Tilde{\varphi}_1-\Tilde{\varphi}_2)|_{S})=\underset{\Om}{\int}\left(\Gr u\Gr(\Tilde{\varphi}_1-\Tilde{\varphi}_2)-f(\Tilde{\varphi}_1-\Tilde{\varphi}_2)\right)\dx-\underset{L}{\int}\frac{\pa u}{\pa n}(\Tilde{\varphi}_1-\Tilde{\varphi}_2)\dsi.$$
     Now, we use Green's formula to get
     \begin{align*}
         \mathcal{F}((\Tilde{\varphi}_1-\Tilde{\varphi}_2)|_{S})&=\underset{\Om}{\int}f(\Tilde{\varphi}_1-\Tilde{\varphi}_2)\dx+\underset{L}{\int}\frac{\pa u}{\pa n}(\Tilde{\varphi}_1-\Tilde{\varphi}_2)\dsi-\underset{\Om}{\int}f(\Tilde{\varphi}_1-\Tilde{\varphi}_2)\dx\\
         &\hspace{4mm}-\underset{L}{\int}\frac{\pa u}{\pa n}(\Tilde{\varphi}_1-\Tilde{\varphi}_2)\dsi\\
         &=0,\\
     \text{which implies }\mathcal{F}(\Tilde{\varphi}_1|_S)&=\mathcal{F}(\Tilde{\varphi}_2|_S).
     \end{align*}
     Hence, $\mathcal{\F}$ is well defined. $\F$ is  linear and continuous on $H^\frac{1}{2}(S)$. \\
     Taking $\varphi\in H^\frac{1}{2}(S)$, we get
     \begin{align*}
         \mathcal{F}(\varphi)&=\underset{\Om}{\int}(\Gr u\Gr \Tilde{\varphi}-f\Tilde{\varphi})\dx-\underset{L}{\int}\frac{\pa u}{\pa n}\Tilde{\varphi}\dsi.
         \end{align*}
         Applying Green's formula, we can write the last result as follows
         \begin{align*}
         \mathcal{F}(\varphi)&=\underset{\Om}{\int}f\Tilde{\varphi}\dx+\underset{L}{\int}\frac{\pa u}{\pa n}\Tilde{\varphi}\dsi+\underset{S}{\int}\frac{\pa u}{\pa n}\Tilde{\varphi}\dsi-\underset{\Om}{\int}f\Tilde{\varphi}\dx-\underset{L}{\int}\frac{\pa u}{\pa n}\Tilde{\varphi}\dsi\\
         &=\underset{S}{\int}\frac{\pa u}{\pa n}\Tilde{\varphi}\dsi\\
         &=\underset{S}{\int}\frac{\pa u}{\pa n}\Big|_S\varphi\dsi\\
         \text{implies } \mathcal{F}&=\frac{\pa u}{\pa n}\Big|_S.
     \end{align*}
     Since $\F$ is a linear and continuous function on $H^\frac{1}{2}(S)$ and $\mathcal{F}=\frac{\pa u}{\pa n}\Big|_S$ , it implies that
     $$\frac{\pa u}{\pa n}\Big|_S\in H^{-\frac{1}{2}}(S).$$
\end{proof}
\begin{definition}
   Let $\Om\subset\R^d$ be a bounded domain. A function $u\in H^1(\Om)$ is said to be weakly harmonic in $\Om$ (or $-\De u=0$ weakly in $\Om$) if 
    $$\underset{\Om}{\int}\Gr u\Gr \Psi\dx=0\hspace{1cm}\text{for all }\Psi\in H_0^1(\Om).$$
    We refer the reader to \cite{Liang} for a discussion of weakly harmonic functions.
\end{definition}
\begin{theorem}\label{Th A.1}
     Let $S\neq\emptyset$. If $u\in H^1(\Om)$ is weakly harmonic in $\Om$ and  satisfies  
   $$u=0\hspace{6mm}\text{on}\hspace{2mm} S,$$
and
    $$ \frac{\pa u}{\pa n}=0\hspace{6mm}\text{on}\hspace{2mm} S,$$
then $u=0$ in $\Om$.
\end{theorem}
\begin{proof}
    See Theorem~1.4 in \cite{mikko}.
\end{proof}
\begin{proposition}\label{Pr A.2}
    $V_D$ and $V_N$ are subspaces of $\mathcal{H}$, where $V_D$, $V_N$, and $\mathcal{H}$ are defined in Section \ref{sec 3.2}. Then $V_D$ and $V_N$ are closed subspaces in $\mathcal{H}$.
\end{proposition}
\begin{proof}
    Let $v$ be the limit point of $V_D$. Take $(v_n)\subset V_D$ such that $$v_n\to v\hspace{4mm} \text{in}\hspace{2mm} \mathcal{H}.$$
    It suffices to show that $v\in V_D$.\\
     For $\phi\in H_0^1(\Om)$ and $v_n\in V_D$, we have
    $$\underset{\Om}{\int}\Gr v_n\Gr \phi\dx=0.$$
    Since $v_n\to v$ in $\mathcal{H}$, it follows that
    $$\underset{\Om}{\int}\Gr v\Gr \phi\dx=0\hspace{6mm}\text{for all }\phi\in H_0^1(\Om).$$
    Hence, $v$ is weakly harmonic.
    Furthermore,  $v_n=0$ on $S$ in the trace sense for every $n\in \N$. By the continuity of the trace operator, it follows that $v=\underset{n\to\infty}{\text{lim}}v_n=0$ on $S$ in the trace sense.\\
    Hence, $V_D$ is a closed subspace in $\mathcal{H}$.\\
    Now, let $w$ be the limit point of $V_N$. Take $(w_n)\subset V_N$ such that $$w_n\to w\hspace{4mm}\text{in}\hspace{2mm} \mathcal{H}.$$
     It suffices to show that $w\in V_N$.\\
    For $\psi\in H_{DL}^1(\Om)$ and $w_n\in V_N$, we have
    $$\underset{\Om}{\int}\Gr w_n\Gr \psi\dx=0.$$
    Since $w_n\to w$ in $\mathcal{H}$, it follows that
    $$\underset{\Om}{\int}\Gr w\Gr \psi\dx=0\hspace{6mm}\text{for all }\psi\in H_{DL}^1(\Om).$$
    Hence, $w$ is weakly harmonic and $\frac{\pa w}{\pa n}=0$ on $S$.\\
    Therefore, $V_N$ is a closed subspace in $\mathcal{H}$.
\end{proof}

\begin{ack}
AM thanks L. Beilina (Chalmers, Sweden) and N. J{\"a}verg{\aa}rd (Karlstad, Sweden) for fruitful discussions around elliptic Cauchy problems and their numerical approximation.
\end{ack}

\begin{funding}
The authors gratefully acknowledge the Swedish Research Council for partial financial support under project "Homogenization of Nonlinear Drift-diffusion Systems for Charge Transport through Bicontinuous Media" (project no. 2024-05606).
 \end{funding}

\bibliographystyle{plain}
\bibliography{references}

\end{document}